\documentclass[reqno, 12pt]{amsart}
\pdfoutput=1
\makeatletter
\let\origsection=\section \def\section{\@ifstar{\origsection*}{\mysection}} 
\def\mysection{\@startsection{section}{1}\z@{.7\linespacing\@plus\linespacing}{.5\linespacing}{\normalfont\scshape\centering\S}}
\makeatother        

\usepackage{amsmath,amssymb,amsthm}
\usepackage{mathrsfs}
\usepackage{mathabx}\changenotsign
\usepackage{dsfont}

\usepackage{xcolor}
\usepackage[backref]{hyperref}
\hypersetup{
    colorlinks,
    linkcolor={red!60!black},
    citecolor={green!60!black},
    urlcolor={blue!60!black}
}

\usepackage{graphicx}

\usepackage{verbatim}

\usepackage[open,openlevel=2,atend]{bookmark}

\usepackage[abbrev,msc-links,backrefs]{amsrefs} 
\usepackage{doi}

\renewcommand{\PrintDOI}[1]{\doi{#1}}

\usepackage[T1]{fontenc}
\usepackage{lmodern}
\usepackage[babel]{microtype}
\usepackage[english]{babel}

\usepackage{geometry}
\numberwithin{equation}{section}
\numberwithin{figure}{section}

\usepackage{enumitem}

\let\polishlcross=\l
\def\l{\ifmmode\ell\else\polishlcross\fi}

\def\paragraph#1{%
  \noindent\textbf{#1.}\enspace}

\let\emptyset=\varnothing
\let\setminus=\smallsetminus

\makeatletter
\def\moverlay{\mathpalette\mov@rlay}
\def\mov@rlay#1#2{\leavevmode\vtop{   \baselineskip\z@skip \lineskiplimit-\maxdimen
   \ialign{\hfil$\m@th#1##$\hfil\cr#2\crcr}}}
\newcommand{\charfusion}[3][\mathord]{
    #1{\ifx#1\mathop\vphantom{#2}\fi
        \mathpalette\mov@rlay{#2\cr#3}
      }
    \ifx#1\mathop\expandafter\displaylimits\fi}
\makeatother

\DeclareFontFamily{U}  {MnSymbolC}{}
\DeclareSymbolFont{MnSyC}         {U}  {MnSymbolC}{m}{n}
\DeclareFontShape{U}{MnSymbolC}{m}{n}{
    <-6>  MnSymbolC5
   <6-7>  MnSymbolC6
   <7-8>  MnSymbolC7
   <8-9>  MnSymbolC8
   <9-10> MnSymbolC9
  <10-12> MnSymbolC10
  <12->   MnSymbolC12}{}
\DeclareMathSymbol{\powerset}{\mathord}{MnSyC}{180}

\let\epsilon=\varepsilon

\let\rho=\varrho
\let\theta=\vartheta

\theoremstyle{plain}
\newtheorem{thm}{Theorem}[section]
\newtheorem{theorem}[thm]{Theorem}
\newtheorem{lemma}[thm]{Lemma}

\newtheorem{abccase}{Case}

\newtheorem{case}{Case}
\newtheorem{subcase}{Subcase}[case]

\newtheorem{subsubcase}{Subcase}[subcase]

\newtheorem{proposition}[thm]{Proposition}

\newtheorem{thm-intro}{Theorem}[]
\newtheorem*{claim*}{Claim}

\theoremstyle{definition}
\newtheorem{definition}[thm]{Definition}

\newtheorem{example}[thm]{Example}
\usepackage{thmtools}
\usepackage{thm-restate}

\usepackage{accents}

\let\phi=\varphi

\newcommand{\no}[1]{}

\def\?#1{\vadjust{\vbox to 0pt{\vss\vskip-8pt\leftline{%
     \llap{\hbox{\vbox{\pretolerance=-1
     \doublehyphendemerits=0\finalhyphendemerits=0
     \hsize16truemm\tolerance=10000\small
     \lineskip=0pt\lineskiplimit=0pt
     \rightskip=0pt plus16truemm\baselineskip8pt\noindent
     \hskip0pt        %(without this, the first word is never hyphenated!)
     {\tiny #1 }\endgraf}\hskip7truemm}}}\vss}}}

\begin{document}

\author[K.~Heuer]{Karl Heuer}
\address{Karl Heuer, Institute of Software Engineering and Theoretical Computer Science, Technische Universit\"{a}t Berlin, Ernst-Reuter-Platz 7, 10587 Berlin, Germany}
\email{\tt karl.heuer@tu-berlin.de}

\author[D.~Sarikaya]{Deniz Sarikaya}
\address{Deniz Sarikaya, Department of Mathematics, University of Hamburg, Bundesstra{\ss}e 55, 20146 Hamburg, Germany}
\email{\tt deniz.sarikaya@uni-hamburg.de}

\title[]{Forcing Hamiltonicity in locally finite graphs via forbidden induced subgraphs II: paws}

\begin{abstract}
In this paper we extend a result about a sufficient condition for Hamiltonicity for finite graphs by Broersma and Veldmann to locally finite graphs.
In order to do this we use topological circles within the Freudenthal compactification of a locally finite graph as infinite cycles.
The condition we focus on in this paper is in terms of forbidden induced subgraphs, namely being claw-free and a relaxation of being paw-free.
\end{abstract}

\maketitle

\section{Introduction}\label{sec:Introduction}

In this second paper out of a series we extend another sufficient condition for Hamiltonicity in finite graphs to locally finite ones.
For this we consider, given a locally finite connected graph $G$, the topological space $|G|$~\cites{Diestel.Buch, Diestel.Arx}, known as the Freudenthal compactification of $G$.
Beside the graph $G$, seen as a $1$-complex, the space $|G|$ also contains additional points, namely the \emph{ends} of $G$, which are equivalence classes of one-way infinite paths of $G$ under the relation of being inseparable by finitely many vertices.
Following the topological approach from~\cites{inf-cyc-1, inf-cyc-2}, we use \emph{circles}, i.e. homeomorphic images of the unit circle $S^1 \subseteq \mathbb{R}^2$ in $|G|$, to extend the notion of cycles and allowing infinite ones.
Then we call $G$ \emph{Hamiltonian} if there is a circle in $|G|$ containing all vertices of $G$.

This series of articles focuses on extending certain local conditions that guarantee the existence of a Hamilton cycle in finite graphs, namely such in terms of forbidden induced subgraphs.
In this paper we focus on a condition involving precisely two graphs: the \emph{claw}, i.e. $K_{1, 3}$, and the \emph{paw}, which is the graph obtained from a triangle and an additional vertex which is adjacent to precisely one vertex of the triangle (cf.~Figure~\ref{paw}).
For the rest of this paper, we shall denote the vertex of degree $1$ in a paw by $a_1$ and those two vertices non-adjacent to $a_1$ by $b_1$ and $b_2$.
The remaining vertex of a paw will always be called $a_0$, as depicted in~Figure~\ref{paw}.

\begin{figure}[htbp]
\centering
\includegraphics[width=8cm]{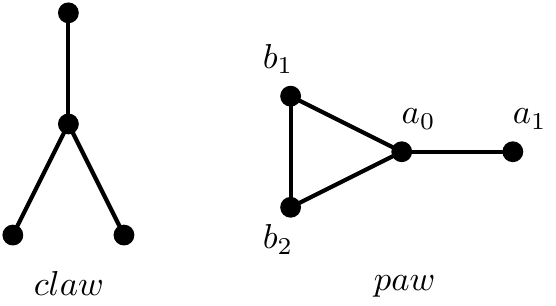} 
\caption{The subgraphs we focus on in this paper.}
\label{paw}
\end{figure}

The following theorem is probably the first Hamiltonicity result for finite graphs in terms of forbidden induced subgraphs.
Note that, given any two graphs $G$ and $H,$ we call $G$ a \emph{$H$-free} graph if $G$ does not contain any induced subgraph isomorphic to $H$.

\begin{theorem}\cite{paw-free}*{Thm.\ 4}\label{thm:paw-free}
Every finite $2$-connected claw-free and paw-free graph is Hamiltonian.
\end{theorem}

Even for finite graphs the condition in Theorem~\ref{thm:paw-free} is very restricting:
the only graphs satisfying this condition are cycles, cliques and cliques with a matching removed.
We shall see in Section~\ref{sec:examples} that there do not exist any infinite locally finite $2$-connected claw-free and paw-free graphs.

However, we shall study a variant of the condition in Theorem~\ref{thm:paw-free} where the paw-freeness is relaxed.
We focus on the following Hamiltonicity result due to Broersma and Veldmann.
In order to state it we have to give two further definitions.
A graph is called \emph{pancyclic} if it contains a cycle of every possible length.
Let $H$ be an induced subgraph of a graph $G$ and $v,w \in V(H)$.
We shall write $\phi_H (v,w)$ for the property that $v$ and $w$ have a common neighbour in $G$ outside of $V(H)$.
We shall simply write $\phi(v,w)$ if the context makes it clear to which subgraph $H$ we are referring to. 

\begin{restatable}{thm}{finBroeVeld}\cite{BroeVeld}*{Thm. 2}\label{thm:finBroe}
Let $G$ be a finite, $2$-connected, claw-free graph.
If every induced paw of $G$ satisfies $\phi(a_1,b_{i})$ for some $i \in \{1, 2\}$, then either $G$ is pancyclic or $G$ is a cycle.
\end{restatable}

Obviously being pancyclic implies Hamiltonicity for finite graphs.
We shall generalise Theorem~\ref{thm:finBroe} to locally finite graphs, where we focus on verifying Hamiltonicity.
We do this since we probably have no meaningful length parameter for distinguishing different infinite cycles, but we have a meaningful notion for Hamiltonicity.
More precisely, we will prove the following:

\begin{restatable}{thm}{infBroeVeld}\label{thm:infBroe}
Let $G$ be a locally finite, $2$-connected, claw-free graph.
If every induced paw of $G$ satisfies $\phi(a_1,b_{i})$ for some $i \in \{1, 2\}$, then $G$ is Hamiltonian.
\end{restatable}

Different from the graphs considered in our first paper of this series~\cite{HC_sub_1}, we shall state examples of graphs fulfilling the premise of Theorem~\ref{thm:infBroe} with arbitrarily, but finitely many ends, with $\aleph_0$ many and with $2^{\aleph_0}$ many ends.

The structure of this paper is as follows.
In Section~\ref{sec:Preliminaries} we introduce the needed definitions and notation.
Then we conclude that section by introducing the tools we need for the proof of Theorem~\ref{thm:infBroe}.
In Section~\ref{sec:examples} we show that no infinite locally finite graphs exist that meet the criteria of Theorem~\ref{thm:paw-free}, except for being finite.
Afterwards we give examples of infinite locally finite graphs fulfilling the conditions of Theorem~\ref{thm:infBroe}.
Finally, we prove our main result, Theorem~\ref{thm:finBroe}, in Section~\ref{sec:mainproof} and we start that section with a sketch of our proof.

\section{Preliminaries}\label{sec:Preliminaries}

In general we follow the graph theoretical notation from \cite{Diestel.Buch}.
Especially regarding the topological notions for locally finite graphs, we refer to \cite{Diestel.Buch}*{Ch.\ 8.5}.
To see a wider survey regarding topological infinite graph theory, see \cite{Diestel.Arx}. 

\subsection{Basic notions}\label{subsec:basic}

All graphs considered in this paper are undirected and simple.
Generally, we do not assume a graph to be finite.
We call a graph \textit{locally finite} if every vertex has finite degree.

For the rest of this section let $G$ denote some graph.
Later in this section, however, we shall make further assumptions on $G$.

Let $X$ be a vertex set of $G$.
We denote by $G[X]$ the induced subgraph of $G$ with vertex set $X$.
For small vertex sets, we sometimes omit the set brackets, i.e.~we write $G[a,b,c]$ as a short form for $G[\{a,b,c\}]$.
We write $G-X$ for the graph $G[V \setminus X]$.
If $H$ is a subgraph of $G$ we shall write $G-H$ instead of $G-V(H)$.
Again we omit set brackets around small vertex sets, especially for singleton sets.
We briefly denote the cut $E(X, V \setminus X)$ by $\delta(X)$.
For any $i \in \mathbb{N}$ we denote by $N_i(X)$ and $N_i(v)$ the set of vertices of distance at most $i$ in $G$ from the vertex set $X$ or from a vertex $v \in V(G)$. 
We denote by $\partial(X)$ the set of vertices $v$ of $X$ with $N(v) \not \subseteq X$.

Let $H$ be a subgraph of $G$ and $v, w \in V(H)$.
We denote by $\varphi_H(v, w)$ the property that $v$ and $w$ have a common neighbour in $G-H$.
We shall drop the subscript when it is clear to which subgraph $H$ we refer to.

Let $C$ be a cycle of $G$ and $u$ be a vertex of $C$.
We implicitly fix an orientation of the cycle and we write $u^+$ and $u^-$ for the neighbour of $u$ in $C$ in positive and negative, respectively, direction of $C$ using a fixed orientation of $C$.
Later on we will not always mention that we fix an orientation for the considered cycle using this notation. 
For two vertices $v$ and $w$ on $C$, we denote by $vCw$ the $v$--$w$ path in $C$ that follows the orientation from $v$ to $w$.

If $G$ is a finite graph containing a cycle of length $s$ for every $s \in \{3,4, \ldots, |V(G)| \}$, we call $G$ \emph{pancyclic}.

If $v$ and $w$ are vertices of a tree $T$, then we denote by $vTw$ the unique $v$--$w$ path in $T$.

A one-way infinite path $R$ in $G$ is called a \textit{ray} of $G$ and a two-way infinite path in $G$ is called a \emph{double ray} of $G$.
A subgraph of a ray $R$ is called a \emph{tail} of $R$.
The unique vertex of degree $1$ of $R$ is called the \emph{start vertex} of $R$.
For a vertex $r$ on a ray $R$, we denote the tail of $R$ with start vertex $r$ by $rR$.

An equivalence relation can be defined on the set of all rays of $G$ by saying that two rays in $G$ are \emph{equivalent} if they cannot be separated by finitely many vertices.
It is easy to check that this defines in fact an equivalence relation.
The corresponding equivalence classes of rays under this relation are called the \textit{ends} of $G$.
We denote the sets of ends of a graph $G$ with $\Omega(G)$.
If $R \in \omega$ for some end $\omega \in \Omega(G)$, then we briefly call $R$ an \emph{$\omega$-ray}.

Note that for any end $\omega$ of $G$ and any finite vertex set $S \subseteq V(G)$ there exists a unique component $C(S, \omega)$ that contains tails of all $\omega$-rays.
We say that a finite vertex set $S \subseteq V(G)$ \emph{separates} two ends $\omega_1$ and $\omega_2$ of $G$ if $C(S, \omega_1) \neq C(S, \omega_2)$.
Note that any two different ends can be separated by a finite vertex set.

Let $R$ be a ray in $G$ and $X \subseteq V(G)$ be finite.
We call $R$ \emph{distance increasing w.r.t. $X$} if $|V(R) \cap N_i(X)| = 1$ for every $i \in \mathbb{N}$.
Note that a distance increasing ray w.r.t.\ $X$ has its start vertex in $X$.

\subsection{Topological notions}\label{subsec:top}

We assume $G$ to be locally finite and connected for the rest of this section.
The graph $G$ together with its ends can be endowed with a certain topology, yielding the space $|G|$ referred to as \emph{Freudenthal compactification} of $G$.
Note that within $|G|$, every ray of $G$ converges to the end of $G$ it is contained in.
For a precise definition of~$|G|$, see \cite{Diestel.Buch}*{Ch.\ 8.5}.
See~\cite{Freud} for Freudenthal's paper about the Freudenthal compactification, and see~\cite{Freud-Equi} about the connection to $|G|$.

Given a point set $X$ in $|G|$, we denote its closure in $|G|$ by $\overline{X}$.

We call the image of a homeomorphism which maps from the unit circle $S^1 \subseteq \mathbb{R}^2$ to $|G|$ a \textit{circle} of $G$.
We call $G$ \textit{Hamiltonian} if there is a circle in $|G|$ containing all vertices of $G$, and thus also all ends of $G$ due the closedness of circles.
Such a circle is called a \emph{Hamilton circle} of $G$.
Note that this notion coincides with the usual notion of Hamiltonicity for finite graphs.

\subsection{Tools}\label{subsec:Tools}

In this subsection we introduce some basic lemmas we shall use to prove our results.
We begin with a brief lemma about the existence of distance increasing rays with respect to finite vertex sets.
The proof of this lemma works via a very easy compactness argument and we omit it here.
In case a proof is desired, see for example~\cite{HC_sub_1}*{Lemma~2.3}.

\begin{lemma}\label{lem:dist-ray}
Let $G$ be an infinite locally finite connected graph and $X \subseteq V(G)$ be finite.
Then there exists a distance increasing ray w.r.t.\ $X$.
\end{lemma}

The following statements are all about claw-free graphs.
The first is a very easy observation and probably folklore, so we do not prove it here.
However, in case a proof is desired, consider for example~\cite{Heuer.2015}*{Prop.\ 3.7.}.

\begin{proposition} \label{prop:2comp} 
Let $G$ be a connected claw-free graph and $S$ be a minimal vertex separator in $G$. Then $G-S$ has exactly two components.
\end{proposition}

Since we have to extend cycles very carefully in the proof of our main result, Theorem~\ref{thm:infBroe}, the following lemma will be very helpful for us.
Again, that result is probably folklore and the proof of that lemma is very easy, but it can be found for example in~\cite{Heuer.2015}*{Lemma 3.8.}.

\begin{lemma}\label{lem:Ncomplete} 
Let $G$ be a connected claw-free graph and $S$ be a minimal vertex separator in~$G$. For every vertex $s \in S$ and every component $K$ of $G-S$, the graph $G[N(s)\cap V(K)]$ is complete.
\end{lemma}

The next lemma is a structural result for locally finite claw-free graphs about vertex sets separating some finite vertex set from all ends of the graph.
This result forms the backbone for the proof of the main result of this article.

\begin{lemma}\cite{Heuer.2015}*{Lemma 3.10}\label{lem:KarlStr}
Let $G$ be an infinite, locally finite, connected, claw-free graph and ${X}$ be a finite vertex set of $G$ such that $G[X]$ is connected. Furthermore, let $\mathfrak{S} \subseteq V(G)$ be a finite minimal vertex set such that $\mathfrak{S} \cap X = \emptyset$ and every ray starting in $X$ has to meet $\mathfrak{S}$. Then the following holds:
\begin{enumerate}
 	\item $G-\mathfrak{S}$ has $k \geq 1$ infinite components $K_1, \ldots, K_k$ and the set $\mathfrak{S}$ is the disjoint union of minimal vertex separators $S_1, \ldots, S_k$ in $G$ such that for every $i$ with $1 \leq i \leq k$ each vertex in $S_i$ has a neighbour in $K_j$ if and only if $j=i$.
 	\item $G-\mathfrak{S}$ has precisely one finite component $K_0$. This component contains all vertices of $X$ and every vertex of $\mathfrak{S}$ has a neighbour in $K_0$.
\end{enumerate}
\end{lemma}

Given a graph $G$, a finite vertex set $X$ and a set $\mathfrak{S}$ all as in Lemma~\ref{lem:KarlStr} we shall call $\mathfrak{S}$ an \emph{$X$-umbrella}.

The following, last lemma of this section is the tool we use to verify Hamiltonicity for locally finite graphs in this paper.

\begin{lemma}\cite{Heuer.2015}*{Lemma 3.11}\label{lem:KarlZiel}
Let $G$ be an infinite, locally finite, connected graph and $(C^i)_{i \in \mathbb{N}}$ be a sequence of cycles of $G$. Now $G$ is Hamiltonian if there exists an integer $k_i \geq 1$ for every $i \geq 1$ and vertex sets $M^i_j \subseteq V(G)$ for every $i \geq 1$ and $j$ with $1 \leq j \leq k_i$ such that the following is true:

\begin{enumerate}
	\item For every vertex $v$ of $G$, there exists an integer $j \geq 0$ such that $v \in V(C^i)$ holds for every $i \geq j$.\label{ziel1}
	\item For every $i \geq 1$ and $j$ with $1 \leq j \leq k_i$, the cut $\delta(M^i_j)$ is finite.\label{ziel2}
	\item For every end $\omega$ of $G$, there is a function $f : \mathbb{N} \setminus \lbrace 0 \rbrace \longrightarrow \mathbb{N}$ such that the inclusion ${M^{j}_{f(j)} \subseteq M^i_{f(i)}}$ holds for all integers $i, j$ with $1 \leq i \leq j$ and the equation ${M_{\omega}:= \bigcap^{\infty}_{i=1} \overline{M^i_{f(i)}} = \lbrace \omega \rbrace}$ is true.\label{ziel3}
	\item $E(C^i) \cap E(C^j) \subseteq E(C^{j+1})$ holds for all integers $i$ and $j$ with $0 \leq i < j$.\label{ziel4}
	\item The equations $E(C^i) \cap \delta(M^p_j) = E(C^p) \cap \delta(M^p_j)$ and $|E(C^i) \cap \delta(M^p_j)| = 2$ hold for each triple $(i, p, j)$ which satisfies $1 \leq p \leq i$ and $1 \leq j \leq k_p$.\label{ziel5}
\end{enumerate}
\end{lemma}

\section{Examples of graph meeting the criteria of Theorem~\ref{thm:infBroe}}\label{sec:examples}

In this section we state examples of infinite locally finite $2$-connected claw-free graphs where each induced paw satisfies $\varphi(a_1, b_i)$ for some $i \in \{ 1, 2 \}$.
While the class of claw-free and net-free graphs, which we considered in the first paper of this series~\cite{HC_sub_1}, allows only graphs with at most two ends, the graphs in this paper have a bigger variety.
We shall give examples of graphs with an arbitrary, but finite number of ends, with $\aleph_0$ many and with $2^{\aleph_0}$ many ends.

However, before we focus on these examples, we prove another proposition.
This result tells us that we cannot try to extend Hamiltonicity results about locally finite $2$-connected claw-free and paw-free graphs as such graphs do not exist.

\begin{proposition}\label{prop:no_paw-free}
Every infinite locally finite connected claw-free graph containing a cycle, also contains an induced paw.
\end{proposition}

\begin{proof}
Let $G$ be a graph as in the statement and let $C$ be a cycle of $G$. 
Pick a ray $R = r_0r_1r_2 \ldots$ that is distance increasing w.r.t.\ $V(C)$, which exists by Lemma~\ref{lem:dist-ray}.
Recall that $r_0 \in V(C)$ holds.
Now $G[r_{0},r_{0}^+,r_{0}^-,r_{1}]$ is no induced claw since $G$ is claw-free.
If $r_{0}^-r_{1} \in E(G)$ or $r_{0}^+r_{1} \in E(G)$, then this would yield a $K_3$, say $G[r_{0}^+,r_{0},r_{1}]$.
Since $R$ is distance increasing w.r.t.\ $V(C)$, we know that $G[r_{0}^+,r_{0},r_{1},r_{2}]$ is an induced paw.

The only other possibility to avoid $G[r_{0},r_{0}^+,r_{0}^-,r_{1}]$ being an induced claw is $r_{0}^-r_{0}^+ \in E(G)$ but $r_{0}^-r_{1}, r_{0}^+r_{1} \notin E(G)$.
Then $G[r_{0}^-,r_{0}^+,r_{0},r_{1}]$ forms an induced paw.
\end{proof}

Before we come to the examples, in which the $k$-blow-up operation is involved, let us recall the definition of a $k$-blow-up.
Given a graph $G$ and some $k \in \mathbb{N}$, we call a graph $G'$ a \emph{$k$-blow-up} of $G$ if we obtain $G'$ from $G$ by replacing each vertex of $G$ by a clique of size $k$, where two vertices of $G'$ are adjacent if and only if they are either both from a common such clique, or the original corresponding vertices were adjacent in $G$. 

In three example we now state graphs that meet the criteria of Theorem~\ref{thm:infBroe} by describing an initial graph, from which we then take the line graph and then a $k$-blow-up.
Before we state the first example, let us fix some notation.
Let $n \in \mathbb{N}$ and let $S_n$ denote the infinite tree where each vertex but one has degree $2$ and the other vertex has degree $n$.

\begin{example}\label{ex:n-ends}
To find a graph with an arbitrary, but finite number of ends, consider the $k$-blow-up of the line graph of $S_n$ where $k \geq 2$ and $n \geq 3$.
To briefly describe this graph in other words:
It is the $k$-blow-up of the graph formed by a complete graph on $n$ vertices $V(K_n) = \{ v_1, \ldots, v_n \}$ with pairwise disjoint rays $R_i$ starting at the $v_i$, see Figure~\ref{n-ends}.

It is immediate that such graphs are $2$-connected and claw-free.
To check that every induced paw satisfies $\varphi(a_1, b_i)$ for some $i \in \{ 1, 2 \}$ is also straightforward, so we leave this to the reader.

\begin{figure}[htbp]
\centering
\includegraphics[width=8cm]{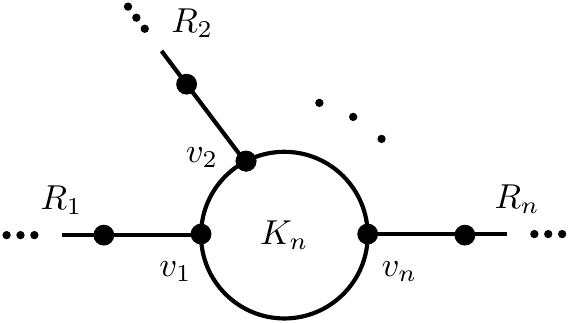} 
\caption{A graph whose $k$-blow-up for $k \geq 2$ has precisely $n \in \mathbb{N}$ ends and meets the conditions of Theorem~\ref{thm:infBroe}.}
\label{n-ends}
\end{figure}

\end{example}

Obviously, we cannot use the construction from example~\ref{ex:n-ends} to obtain suitable graphs with infinitely many ends while staying locally finite.
However, we can extend the idea from the previous example to get such graphs with infinitely many ends.
For every $n \in \mathbb{N}$ let $D_n$ denote the infinite tree where all vertices have degree $2$ except for a set of vertices that induces in $D_n$ a double ray all of whose vertices have degree $n$ within $D_n$.

\begin{example}
Consider the $k$-blow-up of the line graph of $D_n$ for some $k \geq 2$ and $n \geq 3$ (cf.~Figure~\ref{OmegaEnden}).
Again it is easy to verify that such graphs satisfy the conditions of Theorem~\ref{thm:infBroe}.
Furthermore, it is immediate that $D_n$ has precisely $\aleph_0$ many ends for $n \geq 3$ and, hence, so has its line graph and the $k$-blow-up of the line graph.

\begin{figure}[htbp]
\centering
\includegraphics[width=8cm]{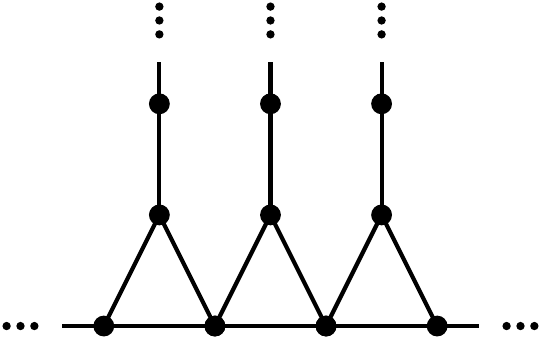} 
\caption{The line graph of $D_3$, whose $k$-blow-up for $k \geq 2$ has precisely $\aleph_0$ many ends and meets the conditions of Theorem~\ref{thm:infBroe}.}
\label{OmegaEnden}
\end{figure}
\end{example}

Let us proceed to the third example describing graphs that meet the conditions of Theorem~\ref{thm:infBroe} and have precisely $2^{\aleph_0}$ many ends.
Again we fix some notation before.
For every $n \in \mathbb{N}$ let $T_n$ denote the infinite tree where each vertex has degree $n$.

\begin{example}
Consider the $k$-blow-up of the line graph $T_n$ for some $k \geq 2$ and $n \geq 3$ (cf.~Figure~\ref{ZweihochOmegaEnden}).
As before, verifying that such graphs satisfy the conditions of Theorem~\ref{thm:infBroe} is easy.
Furthermore, $T_n$ has precisely $2^{\aleph_0}$ many ends for $n \geq 3$.
Therefore, its line graph and the $k$-blow-up of the line graph have that many ends as well.

\begin{figure}[htbp]
\centering
\includegraphics[width=6cm]{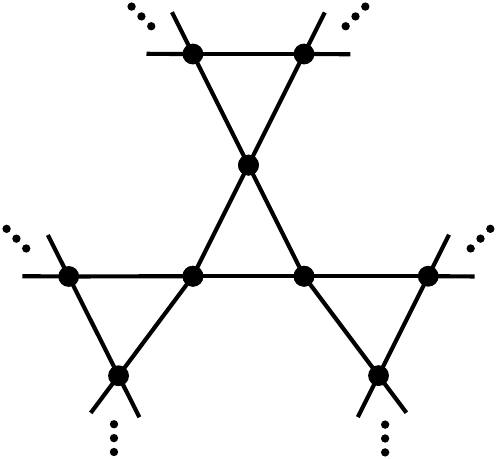} 
\caption{The line graph of $T_3$, whose $k$-blow-up for $k \geq 2$ has precisely $2^{\aleph_0}$ many ends and meets the conditions of Theorem~\ref{thm:infBroe}.}
\label{ZweihochOmegaEnden}
\end{figure}
\end{example}

\section{Proof of Theorem~\ref{thm:infBroe}}\label{sec:mainproof}

In this section we shall prove our main result, Theorem~\ref{thm:infBroe}, which is an extension of Theorem~\ref{thm:finBroe} to locally finite graphs.
Let us describe the general idea for the proof, which has already been successful for other Hamiltonicity results for locally finite claw-free graphs involving local conditions~\cites{Heuer.2015, Heuer.2016}.
In the end we want to apply Lemma~\ref{lem:KarlZiel}.
However, in order to do that, we have to carefully construct suitable cycles and cuts as described in that lemma.
The engine of our proof is the condition about induced paws as in Theorem~\ref{thm:infBroe}.
This allows us to extend any cycle by one or two vertices neighbouring the cycle in a very controlled way keeping most edges of the initial cycle untouched, which is captured in Lemma~\ref{erweiterbar}.

Then the first big step towards the proof of Theorem~\ref{thm:infBroe} is to extend an arbitrary cycle $C$ with respect to a $V(C)$-umbrella $\mathfrak{S}$ to contain all vertices of $K_0$ (as defined in Lemma~\ref{lem:KarlStr}), without containing anything from $\mathfrak{S}$.
This happens in Lemma~\ref{NotTrenner}.
The next step is to carefully extend the cycle into each component $K_i$ while containing all vertices up to the third neighbourhood of $S_i$, but only precisely two vertices of each $S_i$ (cf.~Lemma~\ref{lem:KarlStr}).
This step is rather crucial and achieved in Lemma~\ref{CircleToPromising}.

From this point on we shall not only keep track of a cycle, but also of suitable cuts, each more or less resembling $\delta(V(K_i))$, which our cycle intersects precisely twice.
We shall refer to this as the \emph{$(\star)$ - condition}.
In the remaining lemma, Lemma~\ref{PromisingToGood}, we incorporate all remaining vertices of $\mathfrak{S}$ while maintaining the $(\star)$ - condition.
The key idea here is to dynamically change the cuts as we extend the cycle.
By iterating this whole procedure, we shall get a sequence of cycles and cuts which allows us to apply Lemma~\ref{lem:KarlZiel}.

Not let us start with the lemma which allows us to extend any cycle $C$ within a graph as in Theorem~\ref{thm:infBroe} to incorporate a vertex $v \in N(C)$ without altering many edges of $C$. 

\begin{lemma}\label{erweiterbar}
Let $G$ be a $2$-connected, locally finite claw-free graph such that every induced paw of $G$ satisfies $\phi(a_1,b_{i})$ for some $i \in \{1, 2\}$. 
Let $C$ be a cycle in $G$ and $v$ be a vertex in $N(C)$.
Then there exists a cycle $C'$ and a vertex $w \in N(C)$ such that ${V(C) \cup \{ v\} \subseteq V(C') \subseteq V(C) \cup \{v,w\} }$ holds.
Furthermore, if $xy \in E(C) \Delta E(C')$, then $x \in N_2(v)$ or $y \in N_2(v)$ holds. 
\end{lemma}

\begin{proof}
Let $G$ be a graph as in the statement of the lemma and $C$ be a cycle in $G$. Let $v$ be a vertex in $N(C)$. We prove the statement by a short case distinction. The three cases are depicted in Figure~\ref{pic:ext}:

\begin{case}\label{I-ex}
There is a vertex $u \in V(C) \cap N(v)$ such that $u^+v \in E(G)$ or $ u^- v \in E(G)$.
\end{case}

Without loss of generality let us say $u^+v \in E(G)$.
We obtain $C'$ simply by exchanging $u u^+$ with $uv$ and $vu^+$, which completes Case~1.

\begin{case}
For all vertices $u\in V(C) \cap N(v)$ we have that $u^+ v , u^- v \notin E(G)$.
\end{case}

By the claw-freeness we know that $u^- u^+ \in E(G)$, but this means $G[u^-,u,u^+,v]$ is an induced paw.
Hence we get that $u^-$ or $u^+$ shares a neighbour $w$ with $v$ in ${V(G) \setminus \{u^-,u,u^+,v \}}$, say without loss of generality $u^+$.
Now we distinguish two subcases:

\begin{subcase}\label{IIa-ex}
There exists such a neighbour $w \notin V(C)$.
\end{subcase}

In this situation we simply obtain $C'$ by replacing $uu^+$ by $uv, vw$ and $wu^+$, completing Subcase~2.1.

\begin{subcase}\label{IIb-ex}
All such $w$ lie on $V(C)$.
\end{subcase}

Since we are in the second case, we get that $w^-  w^{+} \in E(G)$.
Hence we can get a new cycle $C'$ by replacing the segment $uCw^+$ of $C$ by $uvwu^{+}Cw^-w^+$.
This completes our case distinction.

Finally, note that in each case changing $C$ to $C'$ does not alter edges whose endvertices have distance at most $2$ from $v$. 
\end{proof}

Given the notation of Lemma~\ref{erweiterbar}, we call the cycle $C'$ a \emph{$v$-extension} of the cycle $C$ \emph{of type $(1)$} if $C'$ is formed as in Case~\ref{I-ex}.
Similarly, we call $C'$ a \emph{$v$-extension} of $C$ \emph{of type $(2.1)$} or \emph{of type $(2.2)$} if $C'$ is formed as in Subcase~\ref{IIa-ex} or~\ref{IIb-ex}, respectively.
For a $v$-extension we also call $v$ the \emph{target} and $u$ its \emph{base}.
We call a cycle $D$ an \emph{extension} of a cycle $C$ if we obtain $D$ from successively performing $v$-extensions (with possibly several different targets~$v$) of $C$ of any type.
Whenever we talk about a $v$-extension of type $(2.1)$ or $(2.2)$, we shall denote by $w$ the same vertex as in the proof of Lemma~\ref{erweiterbar}.
We call the edge we exclude from the cycle $C$ by forming a $v$-extension that has the base as one of its endvertices the \emph{foundation} (of the extension), see Figure~\ref{pic:ext}.

\begin{figure}[htbp]
\centering
\includegraphics[width=15cm]{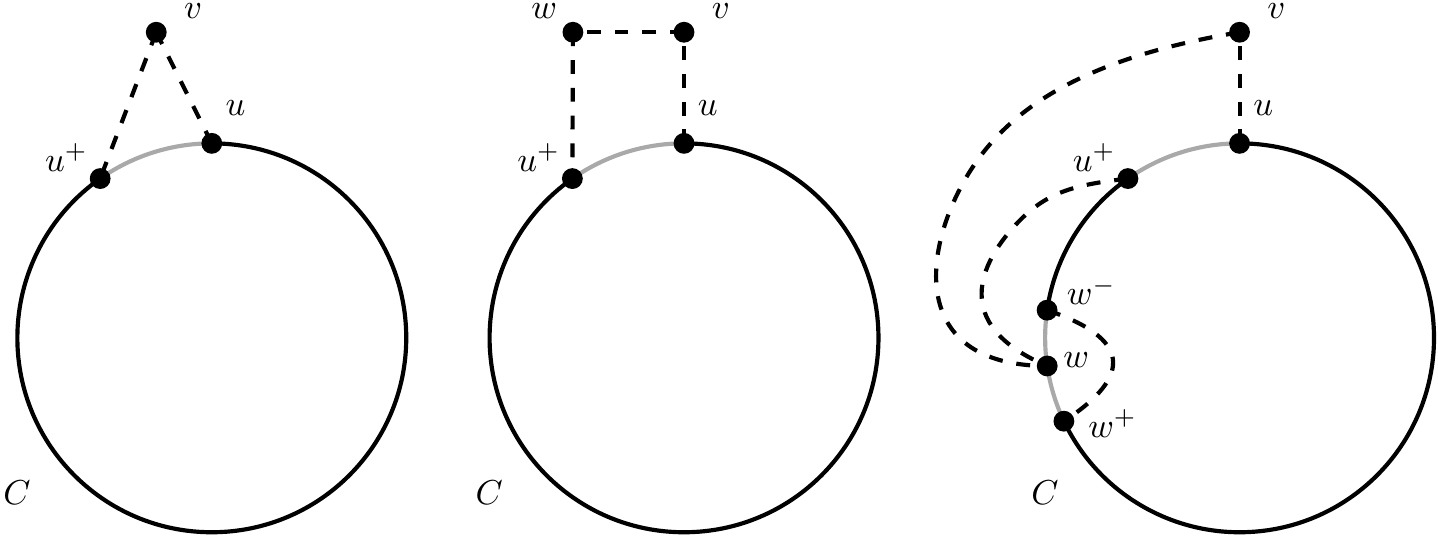} 
\caption{The three types of $v$-extensions of the cycle $C$ by replacing grey edges by dashed ones as occurring in the proof of Lemma~\ref{erweiterbar}.}
\label{pic:ext}
\end{figure}

Given a cycle $C$ and a $V(C)$-umbrella $\mathfrak{S}$ (cf.~Lemma~\ref{lem:KarlStr}), we now show that we can extend $C$ to contain all of $K_0$, but nothing from $\mathfrak{S}$.

\begin{lemma}\label{NotTrenner}
Let $G$ be an infinite, locally finite, $2$-connected, claw-free graph such that every induced paw of $G$ satisfies $\phi (a_1,b_{i} )$ for some $i \in \{ 1, 2 \}$.
Let $C$ be a cycle and $\mathfrak{S}$ a $V(C)$-umbrella.
Then there exists a cycle $C'$ which is an extension of $C$ such that $V(C') = V(K_0)$ and for each $e = xy \in E(C)$ with $x,y \in V(K_0) \setminus N_3 (N(C))$ we have that~$e \in E(C')$. 
\end{lemma}

\begin{proof}
We successively perform $v$-extensions of extensions of $C$ with targets in $V(K_0)$ where, for a fixed target $v$, we always try perform a $v$-extension of type $(1)$.
If that is not possible we try to use one of type $(2.1)$, and if that is not possible either we use one of type $(2.2)$.
The only case which might lead to a problem for the desired equality $V(C') = V(K_0)$ is by using $v$-extensions of type $(2.1)$ since we not only incorporate the target $v$ but one further vertex $w \in N(C)$.
So let us assume we perform a $v$-extension of type $(2.1)$ where $w$ lies in $\mathfrak{S}$.
Thus, $w$ lies in a separator $S_j$ as defined in Lemma~\ref{lem:KarlStr}.
This means that $v$ and $y$, where $uy$ is the foundation of the $v$-extension, are connected since the neighbourhood of $w$ in the component of $G-S_j$ containing $K_0$ forms a clique, due to Proposition~\ref{lem:Ncomplete}.
Hence we did not need to incorporate $v$ via an extension of type $(2.1)$, but could have done it via type $(1)$.
So our process terminates and yields cycle $C'$ such that $V(C') = V(K_0)$  since $K_0$ is finite and connected due to Lemma~\ref{lem:KarlStr}.

Note that for each $e = xy \in E(C)$ such that $x,y \in V(K_0) \setminus N_3 (N(C))$ we have that $e \in E(C')$ holds due to Lemma~\ref{erweiterbar}.
Hence, the cycle $C'$ is as desired.
\end{proof}

Before we move on, we first give a definition used to capture how to carefully extend the cycle obtained from Lemma~\ref{NotTrenner} further into the infinite components $K_1, \ldots, K_k$ as defined in Lemma~\ref{lem:KarlStr} in a convenient way.

\begin{definition}
Let $G$ be an infinite, locally finite, connected, claw-free graph, $C$ be a cycle of $G$ and $\mathfrak{S}$ be a $V(C)$-umbrella.
Furthermore, let $k$, $S_j$ and $K_j$ be defined as in Lemma~\ref{lem:KarlStr}.
Now we call a tuple $(D, M_1, \ldots, M_k)$ \textit{promising} if the following hold for every $j \in \{1, \ldots, k\}$:
\begin{enumerate}
	\item $D$ is a cycle.
	\item $M_j := S_j \cup V(K_j)$.
 	\item $V(K_0) \cup \bigcup_{1 \leq i \leq k} (N_3(S_i) \cap V(K_i)) \subseteq V(D)$   
	\item $\vert E(D) \cap \delta(M_j) \vert = 2$. ($(\star)$ - condition)
\end{enumerate}
\end{definition}

Note that the definition of a promising tuple is defined relative to a certain umbrella.
We shall not mention to which if the context makes this clear.

In the next two lemmas we show how to carefully extend a cycle in two steps to incorporate at least $N_3 (\mathfrak{S})$ while respecting the $(\star)$ - condition.
Now we first prove that promising tuples exist.

\begin{lemma}\label{CircleToPromising}
Let $G$ be an infinite locally finite, connected, claw-free graph such that every induced paw of $G$ satisfies $\phi (a_1,b_{i})$ for some $i \in \{ 1, 2 \}$. 
Furthermore, let $C$ be a cycle of $G$ and $\mathfrak{S}$ be a $V(C)$-umbrella.
Also, let $k$, $S_j$ and $K_j$ be defined as in Lemma~\ref{lem:KarlStr}. 

Then there is a promising tuple $(D, M_1, \ldots, M_k)$ such that for all $e = xy \in E(C)$ with $x,y \in V(K_0) \setminus N_3 (N(C))$ we have $e \in E(D)$.
\end{lemma}

\begin{proof}
Let $C'$ be an extension of $C$ such that $V(C') = V(K_0)$, which exists by Lemma~\ref{NotTrenner}.
We shall successively form $v$-extensions with targets in $\mathfrak{S}$ or replace edges $ab$ in some $S_j$ with $a$--$b$ paths whose inner vertices lie in $K_i$ until we have a cycle $C''$ containing some vertex of each $K_i$, precisely two vertices of each $S_i$, but no edge from any $G[S_i]$, while respecting the {$(\star)$ - condition}.
Let us first prove that having such a cycle $C''$ suffices to prove this lemma. 

Assume we have such a cycle $C''$, which fulfills the $(\star)$ - condition. 
We now show that we can include any finite vertex set $X$ from an arbitrary $K_i$ for $i \in \{1, \ldots, k\}$, hence especially $N_m(S_i) \cap V(K_i)$ for any $m\in \mathbb{N}$ while maintaining the {$(\star)$ - condition}.
For this we only target $v \in V(K_i)$ along a finite connected subgraph within $K_i$ containing $X$.
Hence, the bases of our extensions will always lie in $M_i$.
Whenever possible we include the target via a type $(1)$ extension.
If this is not possible, we try to incorporate the target by a type $(2.1)$ extension and if this is also not possible, we use a type $(2.2)$ extension.

In case we perform a $v$-extension of type $(1)$, no altered edge has an endvertex in $V(K_0)$.
Thus, this does not affect the $(\star)$ - condition.
If, for a $v$-extension of type $(2.1)$, the foundation lies in $M_i$, then again we do not alter any edge with an endvertex in $V(K_0)$.
The other case is a foundation $ur \in \delta(M_i)$ for $r \in \{u^+, u^-\}$.
But this means that we exclude $ur$ from the cut $\delta(M_i)$ and add $rw$ to it.
So again we maintain the $(\star)$ - condition.  
Finally let us consider a $v$-extension of type $(2.2)$.
For $w \in V(K_i)$, we do not affect the $(\star)$ - condition.
So let us assume $w$ lies in the separator $S_i$.
As for an extension of type $(2.1)$, the removal of the foundation $ur$ for $r \in \{u^+, u^-\}$ and incorporating the edges $uv, vw$ and $wr$ does not change the size of the cut.
It remains to check that excluding $w^-w$ and $ww^+$ and adding $w^-w^+$ does not violate the $(\star)$ - condition either.
We check this by a short case distinction:

\begin{abccase}
$w^-, w^+ \in M_i$.
\end{abccase}

In this case the edges $w^-w, ww^+$ and $w^-w^+$ all lie in $M_i$, so the $(\star)$ - condition is maintained.

\begin{abccase}
$w^-, w^+ \in V(G) \setminus M_i$.
\end{abccase}

This case cannot happen since we assumed that $C''$ contains a vertex of $K_i$ and precisely two vertices from $S_i$.
However, $C''$ would cross $\delta(M_i)$ via the edges $w^-w$ $ww^+$ without entering $K_i$.
So the $(\star)$ - condition would already be violated by $C''$; a contradiction.

\begin{abccase}
$w^- \in M_i$ but $w^+ \notin M_i$ (or vice versa).
\end{abccase}

In this case our $v$-extension does not intersect $\delta(M_i)$ with the edge $w^+ w$ anymore, but with the edge $w^-w^+$.
Hence, the $(\star)$ - condition is again maintained.
This completes our case analysis and shows the existence of a desired tuple $(D, M_1, \ldots, M_k)$ if the cycle $C''$ exists.
Note that for each $e = xy \in E(C)$ with $x,y \in V(K_0) \setminus N_3 (N(C))$ we have $e \in E(D)$, as shown in Lemma~\ref{erweiterbar}.

Hence it remains to show that a cycle $C''$ exist, i.e., a cycle containing all of $V(K_0)$, at least some vertex of $K_i$ and precisely two vertices from each $S_i$, but no edge from any $G[S_i]$, while respecting the $(\star)$ - condition.
Say we already have a cycle $Z$ satisfying the following for a (possibly empty) subset $I \subseteq \{ 1, \ldots, k \}$.
\begin{enumerate}[label=(\roman*)]
    \item $V(K_0) \subseteq V(Z)$.\label{con1}
    \item $V(Z) \cap V(K_i) \neq \emptyset$ for every $i \in I$.\label{con2}
    \item $|V(Z) \cap S_i| = 2$, but $E(Z) \cap E(G[S_i]) = \emptyset$ for every $i \in I$.\label{con3}
    \item $\vert E(Z) \cap \delta(M_i) \vert = 2$ for every $i \in I$.\label{con4}
    \item $V(Z) \cap M_j = \emptyset$ for every ${j \in \{1, \ldots, k \} \setminus I}$.
\end{enumerate}

Note that $C'$ meets these five conditions with $I = \emptyset$.
Next we show how to obtain a cycle meeting all these five conditions for a superset of $I$.
We form a $v$-extension $Z'$ of $Z$ with target $v \in S_j$ for some ${j \in \{1, \ldots, k \} \setminus I}$.
Again we analyse the situation via a case distinction.

\setcounter{case}{0}
\begin{case}
$Z'$ can be formed of type $(1)$.
\end{case}

Without loss of generality say the foundation of $Z'$ is $uu^+$.
If $uu^+ \in E(K_0)$, then condition~\ref{con4} from above is still maintained.
If $uu^+ \in \delta(M_i)$ for some $i \in I$, the size of the intersection of the cycle with the cut $\delta(M_j)$ stays the same, but the edge $uv$  meets it instead of $uu^+$.
In both cases $Z'$ now fulfills condition~\ref{con4} for $I \cup \{ j \}$.
Next we form a $v'$-extension $Z''$ of $Z'$ where $v' \in N(v) \cap V(K_j)$.
Such a $v'$ exists since $S_j$ is a minimal vertex separator by Lemma~\ref{lem:KarlStr}.
But now $v'$ can only be included as in type $(2.1)$, since extensions of type $(1)$ and $(2.2)$ need two neighbours of the target on the cycle from which the extension is formed.
Hence, $w$ lies in $S_j$, again we do not change the size of any intersection $ E(Z') \cap \delta(M_i) $ with $i \in \{1, \ldots, k \}$ and we do not incorporate any edge within~$G[S_j]$.
So $Z''$ satisfies all five condition above with respect to $I \cup \{ j \}$.
This completes the analysis of the first case.

Note that in the situation where we cannot perform a $v$-extension of $Z$ of type $(1)$ we can, by the proof of Lemma~\ref{erweiterbar}, fix a desired base $u \in N(v)$ on $Z$ in advance.

\begin{case}
$Z'$ cannot be formed of type $(1)$, but of type $(2.1)$ with base $u \in V(K_0)$.
\end{case}

Now we further distinguish two subcases.

\begin{subcase}\label{sub2.1}
$w \in S_j$.
\end{subcase}

In this case we have incorporated two vertices of $S_j$, meet condition~\ref{con4} also for $j$, but we use the edge $vw$ from $G[S_j]$.
As $S_j$ is a minimal vertex separator by Lemma~\ref{lem:KarlStr}, both of these vertices have a neighbour in $K_j$.
Using that $K_j$ is connected, we can find a $v$--$w$ path $P$ whose inner vertices lie in $K_j$.
Now we modify $Z'$ by replacing $vw$ with $P$.
The resulting cycle is as desired.

\begin{subcase}
$w \in S_i$ for some $i \neq j$.
\end{subcase}

For the sake of clarity we again distinguish two further subcases depending on where the foundation of the $v$-extension lies.
Let us denote the foundation by $ur$ where $r\in \{u^-, u^+\}$.

\begin{subsubcase}
$r \notin S_i$.
\end{subsubcase}

This is not possible since the neighbourhood of $w$ in each component of $G-S_i$ induces a clique.
Since we demand for $u$ as the base of $Z'$ to lie in $V(K_0)$, we know that both $u$ and $r$ lie in the same component of $G-S_i$, namely the one different from $K_i$.
Hence, $vr \in E(G)$ and we could have included $v$ as in type $(1)$ against our assumption.

\begin{subsubcase}\label{sub2.2.2}
$r \in S_i$.
\end{subsubcase}

In this case we know that $i \in I$.
So $ur \in \delta(M_i)$ holds.
By condition~\ref{con3} we know that $|V(Z) \cap S_i| = 2$.
Without loss of generality, say ${V(Z) \cap S_i = \{ r, s \}}$ for some other vertex $s \in S_i$.
As $Z$ satisfies condition~\ref{con1}, \ref{con2} and \ref{con3}, we know that $Z$ contains precisely one path $Q$ with endvertices in $S_i$ and all inner vertices in $K_i$.
Hence, $Q$ must be an $s$--$r$-path.
Now form the $v$-extension $Z'$ of $Z$ of type $(2.1)$, then delete all inner vertices of $Q$ and $r$ from $Z'$ and replace it by an $s$--$w$-path all whose inner vertices lie in $K_i$.
The resulting cycle contains only $v$ from $S_j$ and we can proceed as in Case~1.
This completes the analysis for the second case.

\begin{case}
$Z'$ cannot be formed of type $(1)$, but of type $(2.2)$ with base $u \in V(K_0)$.
\end{case}

Note that in this case $w \in V(Z)$ and $w$ does lie in $S_i$ with $i \in I$ as due to condition~\ref{con2} of $Z$ the edge $w^- w^+$ would cross the separator $S_i$; a contradiction.
Hence, $w \in V(K_0)$ holds.
The only possibility that ${E(Z') \cap \delta(M_{\ell}) \neq E(Z) \cap \delta(M_{\ell})}$ for some ${\ell} \in I$ is that one of $w^+, w^-$ is contained in $S_{\ell}$ while the other vertex is contained in the component of $G-S_{\ell}$ different from $K_{\ell}$.
Say $w^+ \in S_{\ell}$ holds.
Note that not both, $w^+$ and $w^-$ can lie in $S_{\ell}$ since $V(C) \subseteq V(K_0) \cap V(Z')$.
Hence, $|E(Z') \cap \delta(M_{i})| = 2$ holds for all $i \in I \cup \{ j \}$ and $Z'$ contains precisely $v$ from $S_j$.
To complete the argument for this case we can now proceed as in Case~1.
This completes our case analysis and shows the existence of the desired cycle~$C''$.

Finally, note that the constructed promising tuple $(D, M_1, \ldots, M_k)$ satisfies that for each $e = xy \in E(C)$ with $x,y \in V(K_0) \setminus N_3 (N(C))$ we have $e \in E(D)$.
We either performed $v$-extensions, which cause no problems as checked in Lemma~\ref{erweiterbar}.
Apart from that we included or excluded paths from $V(K_i) \cup S_i$ in Subcase~\ref{sub2.1} and Subcase~\ref{sub2.2.2}, which does not affect edges from $C$.
\end{proof}

From Lemma~\ref{CircleToPromising} we do not necessarily get a cycle that contains all vertices of $\mathfrak{S}$.
However, in order to apply Lemma~\ref{lem:KarlZiel} we have to incorporate all these vertices while maintaining constraints for suitable cuts $\delta(M_i)$.
With the next lemma we shall achieve this.
The key idea is to start from a promising tuple and incorporate the remaining vertices while dynamically changing the vertex sets $M_i$ to maintain the constraints for the cuts~$\delta(M_i)$.
We shall encode our desired objects via the following definition before we move on to the next lemma.

\begin{definition}
Let $G = (V, E)$ be an infinite locally finite, connected, claw-free graph, $C$ be a cycle of $G$ and $\mathfrak{S}$ be a $V(C)$-umbrella. Furthermore, let $k$, $S_j$ and $K_j$ be defined as in Lemma~\ref{lem:KarlStr}. Now we call a tuple $(D, M_1, \ldots, M_k)$ \textit{good} if the following properties hold for every $j \in \{1, \ldots, k\}$:
\begin{enumerate}
	\item $D$ is a cycle of $G$ which contains $V(K_0)$ and $S_{j} \cup (N_3(S_{j}) \cap V(K_{j}))$.\label{prop1}
	\item $ {V(K_j) \setminus N(S_j) \subseteq M_j \subseteq V(K_j) \cup \mathfrak{S} \cup N ( \mathfrak{S})}$.\label{prop2}
	\item ${|E(D) \cap \delta(M_{j})| = 2}$. ($(\star)$ - condition)
\end{enumerate}
\end{definition}

\begin{lemma}\label{PromisingToGood}
Let $G$ be an infinite locally finite, connected, claw-free graph, such that every induced paw of $G$ satisfies $\phi (a_1,b_{i})$ for some $i \in \{ 1, 2 \}$.
Furthermore, let $C$ be a cycle of $G$ and $\mathfrak{S}$ be a $V(C)$-umbrella, where $k$, $S_j$ and $K_j$ be defined as in Lemma~\ref{lem:KarlStr}.
Let $(D, M_1, \ldots, M_k)$ be a promising tuple for the $V(C)$-umbrella $\mathfrak{S}$.

Then there is a good tuple $(D', N_1, \ldots, N_k)$ (for the same umbrella $\mathfrak{S}$) and for each $e = xy \in E(C)$ such that $x,y \in V(K_0) \setminus N_3 (N(C))$ we have $e \in E(D')$.
\end{lemma} 

\begin{proof}
We prove the statement via a recursive construction always performing $v$-extensions with targets $v \in \mathfrak{S} \setminus V(D)$ until we eventually obtain the cycle $D'$ of our desired good tuple.
We initialise this construction with the promising tuple $(D, M_1, \ldots, M_k)$.
Let us denote by ${ }^pD$ the cycle obtained after $p$ many performed $v$-extensions.
During this process we also alter the sets $M_i$.
Let ${ }^pM_i$ denote the corresponding vertex set after having performed $p$ many $v$-extensions for every $i \in \{ 1, \ldots, k \}$.
We note that this construction process will eventually terminate, say after $z \leq \vert \mathfrak{S} \setminus V(D) \vert$ many steps, since in each step we add at least one vertex from the finite set $\mathfrak{S} \setminus V(D)$ and do not exclude any vertices at all. 

Whenever possible we include our target vertices via a type $(1)$ extension. If this is not possible, we try to incorporate them via a type $(2.1)$ extension and if this is also not possible, we use an extension of type $(2.2)$. 

Now suppose we have already constructed ${ }^pD$ and ${ }^pM_i$ for all $i \in \{ 1, \ldots, k \}$ such that the $(\star)$ - condition is maintained and let $v \in S_i \setminus V({ }^pD)$ be our next target. 

\setcounter{case}{0}
\begin{case}
There exists a $v$-extension ${ }^{p+1}D$ of ${ }^pD$ of type $(1)$.
\end{case}

Let $xy$ be the foundation of ${ }^{p+1}D$.
Recall that this edge gets substituted by $xv$ and $vy$ by forming the extension.
We define for all $r \in \{1, \ldots, k\}$ the set
\begin{equation*}
	{ }^{p+1}M_r=\begin{cases}
		{ }^{p}M_r    \setminus \{v\}	& \text{ if } x, y \notin { }^pM_r    \\
		{ }^{p}M_r  \cup \{v\}			&\text{ else }\\
	\end{cases}
\end{equation*}
 
This ensures that the $(\star)$ - condition is still maintained for ${ }^{p+1}D$ and each ${ }^{p+1}M_r$.

\begin{case}
There exists no $v$-extension of ${ }^pD$ of type $(1)$, but of type $(2.1)$.
\end{case}

Let $ux$ be the foundation of ${ }^{p+1}D$ where $u$ is the base of the extension.  
Now we define for all $r \in \{1, \ldots, k\}$:

\begin{equation*}
	{ }^{p+1}M_r = 
	\begin{cases}
		{ }^{p}M_r  	\cup \{v,w\}   	&\text{if $u,x \in { }^{p}M_r $}\\
		{ }^{p}M_r  \setminus \{w\} 	\cup \{v\} 	&\text{if $u \in { }^{p}M_r$ and  $x \not \in { }^{p}M_r$}\\
		{ }^{p}M_r  \setminus \{v \}		\cup\{w\}		&\text{if $u \not\in { }^{p}M_r$  and $x \in { }^{p}M_r $}\\
		{ }^{p}M_r  \setminus \{v,w\}		&\text{if $u,x \not\in { }^{p}M_r $}
	\end{cases}
\end{equation*}

\vspace{7pt}

Again we note that these cases respect the $(\star)$ - condition: If $u$ and $x$ lie within one ${ }^pM_r$, the cycle ${ }^{p+1}D$ meets $\delta({ }^{p+1} M_r)$ still twice.
If the $ux \in \delta({ }^pM_s)$, we exclude $ux$ from the corresponding intersection but add precisely $vw$ to it.

\begin{case}
There exists no $v$-extension of ${ }^pD$ of type $(1)$, but of type $(2.2)$.
\end{case}

Again let $ux$ be the foundation of ${ }^{p+1}D$. 
Here we define the new sets for $r \in \{1, \ldots, k\}$ as in the case of an type $(2.1)$ $v$-extension:

\begin{equation*}
	{ }^{p+1}M_r = 
	\begin{cases}
		{ }^{p}M_r \cup \{v,w\}   		&\text{if $u,x \in { }^{p}M_i $}\\
		{ }^{p}M_r \setminus \{w\} 	\cup\{v\}			&\text{if $u \in { }^{p}M_i$ and  $x \not \in { }^{p}M_i $}\\
		{ }^{p}M_r  \setminus\{v\} 	\cup\{w\}			&\text{if $u \not\in { }^{p}M_i$  and $x \in { }^{p}M_i $}\\
		{ }^{p}M_r \setminus \{v,w\}					&\text{if $u,x \not\in { }^{p}M_i $}
	\end{cases}
\end{equation*}

\vspace{7pt}

To verify that the new cuts $\delta({ }^{p+1}M_r)$ together with ${ }^{p+1}D$ still satisfy the $(\star)$ - condition, we first note that we delete three edges from the cycle, namely: $w^-w, ww^+$ and $ux$.
But we include the edges $uv, vw, wx, w^-w^+$.
Regarding the intersection of the cycle and the cuts, the exclusion of $w^-w$ and $ww^+$ precisely cancels out the effect of adding $w^-w^+$. 
The same holds for excluding $ux$ and adding $uv, vw$ and $wx$ as we already in Case~2.
This completes the recursive definition of the cuts for the good tuple. 

We now define $D' := { }^{z}D$ and $N_r := { }^{z}M_r$ for every $r \in \{1, \ldots, k\}$.
It remains to check that $(D', N_1, \ldots, N_k)$ satisfies all properties of a good tuple.
We already argued that the tuple respects the $(\star)$ - condition.

Regarding property~(\ref{prop1}) note that we we started the recursive definition with a cycle $D$ that is part of a promising cycle for the $V(C)$-umbrella $\mathfrak{S}$.
So $D$ already contained all vertices from $V(K_0) \cup (N_3(S_{j}) \cap V(K_{j}))$ for every $j \in \{1, \ldots, k\}$, and we just included the remaining vertices from $ \mathfrak{S}$ when forming $D'$.

Note for property~(\ref{prop2}) that we have only added or excluded vertices in $\mathfrak{S} \cup N(\mathfrak{S})$ in each step when changing the vertex sets for our cuts.
Also note that the vertex set $M_j$ we started with were defined as $M_j = S_j \cup V(K_j)$.
\end{proof}

We now combine the previous lemmas to prove the main theorem.
Let us restate the statement of theorem first.

\infBroeVeld*

\begin{proof}
Let $G$ be a graph as in the statement of the theorem.
We may assume $G$ to be infinite by Theorem~\ref{thm:finBroe}. 
We shall recursively construct a sequence of good tuples $(C^i, M^i_1, \ldots, M^i_{k(i)})$ where each tuple $(C^{i+1}, M^{i+1}_1, \ldots, M^{i+1}_{k({i+1})})$ is defined with respect to a $V(C^i)$-umbrella for every $i \in \mathbb{N}$ as follows.

Start with an arbitrary cycle $A$ in $G$.
This is possible, since $G$ is 2-connected.
Let $C^0$ by an extension of $A$ with $N_3(V(A)) \subseteq V(C^0)$.

Next suppose the cycle $C^i$ has already been defined up to some $i \in \mathbb{N}$:

\begin{itemize}
	\item Let $\mathfrak{S}^{i+1}$ be a $V(C^i)$-umbrella, $K^{i+1}_i$ and $S^{i+1}_i$ be defined as in the Lemma~\ref{lem:KarlStr} and let $k: \mathbb{N} \rightarrow \mathbb{N}$ be the function such that $\mathfrak{S}^{i+1}$ leaves precisely $k(i+1)$ infinite components.
	\item Let $(D^{i+1}, Y^{i+1}_1, \ldots, Y^{i+1}_{k({i+1})})$ be a promising tuple we get by applying Lemma~\ref{CircleToPromising} with the cycle $C^i$ and the $V(C^i)$-umbrella $\mathfrak{S}^{i+1}$.
	\item Then set $(C^{i+1}, M^{i+1}_1, \ldots, M^{i+1}_{k({i+1})})$ to be a good tuple we get from Lemma~\ref{PromisingToGood} applied with the promising tuple $(D^{i+1}, Y^{i+1}_1, \ldots, Y^{i+1}_{k({i+1})})$, the cycle $C^i$ and the $V(C^i)$-umbrella $\mathfrak{S}^{i+1}$.
\end{itemize}

We now conclude the proof by verifying that we can apply Lemma~\ref{lem:KarlZiel}.

For condition~(\ref{ziel1}) of Lemma~\ref{lem:KarlZiel} we have to show that for every $v \in V(G)$ there is some $j \in \mathbb{N}$ such that $v \in V(C^i)$ for every $i \geq j$.
This holds since every $v \in V(G)$ has finite distance to $V(C^0)$.
So it follows from property~(\ref{prop1}) of good tuples.

For condition~(\ref{ziel2}) of Lemma~\ref{lem:KarlZiel} we need to prove that for every $i \geq 1$ and $j$ with $1 \leq j \leq k(i)$, the cut $\delta(M^i_j)$ is finite.
Since $G$ is locally finite, it suffices to show that $M^i_j$ has a finite neighbourhood.
Due to property~(\ref{prop2}) of good tuples we know that $N(M^i_j) \subseteq \mathfrak{S}^i \cup N_2(\mathfrak{S}^i)$.
Since $\mathfrak{S}^i$ is a finite set and $G$ is locally finite, we obtain that $N(M^i_j)$ is finite. 

Regarding condition~(\ref{ziel3}) of Lemma~\ref{lem:KarlZiel} we need to prove for every end $\omega$ of $G$ the existence of a function $f : \mathbb{N} \setminus \lbrace 0 \rbrace \longrightarrow \mathbb{N}$ such that the ${M^{j}_{f(j)} \subseteq M^i_{f(i)}}$ holds for all integers $i, j$ with $1 \leq i \leq j$ and that the equation ${M_{\omega}:= \bigcap^{\infty}_{i=1} \overline{M^i_{f(i)}} = \lbrace \omega \rbrace}$ is true.
To verify this let us fix an arbitrary end $\omega$ of $G$.
We first define the desired function $f$.
Note that for each $i \geq 1$ we know that $K^i_0$ and $\mathfrak{S}^i$ are finite by Lemma~\ref{lem:KarlStr}.
Hence each $\omega$-ray has a tail in $K^i_{\ell}$ for some $\ell \in \{ 1, \ldots, k(i) \}$.
Now set $f(i) := \ell$.

Let us now check that 
$M^{j}_{f(j)} \subseteq M^i_{f(i)}$ holds for all $1 \leq i \leq j$.
By properties~(\ref{prop1}) and (\ref{prop2}) of a good tuple, it is easy to see that $V(K_{f(j)}^j) \subseteq V(K_{f(i)}^i)$ holds for all $1 \leq i \leq j$ and similarly $M^{j}_{f(j)} \subseteq M^i_{f(i)}$.

To verify condition~$(\ref{ziel3})$ it remains to show that ${M_{\omega}:= \bigcap^{\infty}_{i=1} \overline{M^i_{f(i)}} = \lbrace \omega \rbrace}$ is true.
First note that no vertex can lie in this intersection since each vertex will eventually be contained in $K^t_0$ for some sufficiently large $t \in \mathbb{N}$, and hence not in any set $M^{\ell}_i$ for all $\ell \geq t$.
Furthermore, the definition of $f$ already ensures $\omega \in M_{\omega}$.
So let us consider some end $\omega'$ of $G$ distinct from $\omega$.
Let $S$ be a finite vertex set separating $\omega$ from $\omega'$.
Since $S$ is finite, we know that $S \subseteq V(K^t_0)$ holds for some sufficiently large $t \in \mathbb{N}$.
Hence $\omega' \notin \overline{K^t_{f(t)}}$ and similarly $\omega' \notin \overline{M^t_{f(t)}}$.
So we obtain the desired conclusion $\omega' \notin M_{\omega}$

Now we focus on condition~$(\ref{ziel4})$ of Lemma~\ref{lem:KarlZiel}.
There we need to verify the inclusion ${E(C^i) \cap E(C^j) \subseteq E(C^{j+1})}$ for all integers $i$ and $j$ with $0 \leq i < j$. 
This holds since we checked in Lemma~\ref{NotTrenner}, Lemma~\ref{CircleToPromising} and Lemma~\ref{PromisingToGood} that whenever we change a cycle $C$ to a cycle $C'$ each edge $e = xy \in E(C)$ with $x,y \in V(K_0) \setminus N_3 (N(C))$ lies also in $E(C')$.
So by property~\ref{prop1} of good tuples the sequence of cycles $(C^i)_{i \in \mathbb{N}}$ satisfies condition~$(\ref{ziel4})$ of Lemma~\ref{lem:KarlZiel}.

Finally, let us verify condition~$(\ref{ziel5})$ of Lemma~\ref{lem:KarlZiel}.
So we have to show that the equations ${E(C^i) \cap \delta(M^p_j) = E(C^p) \cap \delta(M^p_j)}$ and $|E(C^i) \cap \delta(M^p_j)| = 2$ hold for each triple $(i, p, j)$ which satisfies $1 \leq p \leq i$ and $1 \leq j \leq k(p)$.
This, however, immediately follows from the satisfied previous~$(\ref{ziel4})$ and the fact that good tuples satisfy the $(\star)$ - condition.

Hence, we can apply Lemma~\ref{lem:KarlZiel}, which proves the Hamiltonicity of $G$ and concludes our proof. 
\end{proof}

\section*{Acknowledgements}
Karl Heuer was supported by a postdoc fellowship of the German Academic Exchange Service (DAAD) and by the European Research Council (ERC) under the European Union's Horizon 2020 research and innovation programme (ERC consolidator grant DISTRUCT, agreement No.\ 648527).

Deniz Sarikaya is thankful for the financial and ideal support of the Studienstiftung des deutschen Volkes and the Claussen-Simon-Stiftung.

Furthermore, both authors would like to thank Max Pitz for a comprehensive feedback on an early draft of this paper.
Also they would like to thank Hendrik Niehaus and J.~Pascal Gollin for helpful comments on an early version of this article.

\end{document}